\documentclass[10pt,reqno]{amsart}
\usepackage[hypertex]{hyperref}

\usepackage{longtable}
\usepackage{amsmath,amsthm}
\usepackage{amssymb}
\usepackage{euscript}

\numberwithin{equation}{subsection}

\newcommand{\sqsp}{\renewcommand{\baselinestretch}{1.1}\tiny\normalsize}

\newtheorem{theorem}[subsection]{Theorem}
\newtheorem{lemma}[subsection]{Lemma}

\newtheorem{corollary}[subsection]{Corollary}

\theoremstyle{definition}

\newtheorem{example}[subsection]{Example}

\newcommand{\bk}{\mathbf{k}}
\newcommand{\bC}{\mathbf{C}}

\def\Tone{{\begin{pmatrix}0 & 0 \\ 0 & 0\end{pmatrix}}}
\def\Tonea{{\begin{pmatrix}a_1 & b_1 \\ a_2 & b_2\end{pmatrix}}}

\def\Ttwo{{\begin{pmatrix}e_2 & 0 \\ 0 & 0\end{pmatrix}}}
\def\Ttwoa{{\begin{pmatrix}a_1 & 0 \\ a_2 & a_1^2\end{pmatrix}}}
\def\TtwoaM{{\begin{pmatrix}a_1^2e_2 & 0 \\ 0 & 0\end{pmatrix}}}

\def\Tthree{{\begin{pmatrix}0 & 0 \\ -e_1 & 0\end{pmatrix}}}
\def\Tthreea{{\begin{pmatrix}0 & b_1 \\ 0 & b_2\end{pmatrix}}}

\def\Tthreeb{{\begin{pmatrix}a & 0 \\ 0 & 1\end{pmatrix}}}
\def\TthreebM{{\begin{pmatrix}0 & 0 \\ -ae_1 & 0\end{pmatrix}}}

\def\None{{\begin{pmatrix}e_1 & 0 \\ 0 & e_2\end{pmatrix}}}
\def\Nonea{{\begin{pmatrix}a_1 & b_1 \\ a_2 & b_2\end{pmatrix}}}
\def\NoneaM{{\begin{pmatrix}a_1e_1+a_2e_2 & 0 \\ 0 & b_1e_1+b_2e_2\end{pmatrix}}}

\def\Ntwo{{\begin{pmatrix}e_1 & 0 \\ 0 & 0\end{pmatrix}}}
\def\Ntwoa{{\begin{pmatrix}a & 0 \\ 0 & b\end{pmatrix}}}
\def\NtwoaM{{\begin{pmatrix}ae_1 & 0 \\ 0 & 0\end{pmatrix}}}

\def\Nthree{{\begin{pmatrix}e_1 & e_2 \\ e_2 & 0\end{pmatrix}}}
\def\Nthreea{{\begin{pmatrix}1 & 0 \\ 0 & b\end{pmatrix}}}
\def\NthreeaM{{\begin{pmatrix}e_1 & be_2 \\ be_2 & 0\end{pmatrix}}}

\def\Nfour{{\begin{pmatrix}0 & e_1 \\ 0 & e_2\end{pmatrix}}}
\def\Nfoura{{\begin{pmatrix}a & b \\ 0 & 1\end{pmatrix}}}
\def\NfouraM{{\begin{pmatrix}0 & ae_1 \\ 0 & be_1+e_2\end{pmatrix}}}

\def\Nfive{{\begin{pmatrix}0 & e_1 \\ 0 & e_1+e_2\end{pmatrix}}}
\def\Nfivea{{\begin{pmatrix}1 & b \\ 0 & 1\end{pmatrix}}}
\def\NfiveaM{{\begin{pmatrix}0 & e_1 \\ 0 & (1+b)e_1+e_2\end{pmatrix}}}

\def\Nsix{{\begin{pmatrix}0 & e_1 \\ \lambda e_1 & e_2\end{pmatrix}}}
\def\Nsixa{{\begin{pmatrix}a & 0 \\ 0 & 1\end{pmatrix}}}
\def\NsixaM{{\begin{pmatrix}0 & ae_1 \\ \lambda ae_1 & e_2\end{pmatrix}}}

\def\Aonea{{\begin{pmatrix} a_1 & a_2 & a_3 \\ b_1 & b_2 & b_3 \\ c_1 & c_2 & c_3 \end{pmatrix}}}

\def\Atwo{{\begin{pmatrix} 0 & 0 & 0 \\ 0 & 0 & 0 \\ 0 & 0 & e_1 \end{pmatrix}}}
\def\Atwoa{{\begin{pmatrix} c_3^2 & b_1 & c_1 \\ 0 & b_2 & c_2 \\ 0 & 0 & c_3 \end{pmatrix}}}
\def\AtwoaM{{\begin{pmatrix} 0 & 0 & 0 \\ 0 & 0 & 0 \\ 0 & 0 & c_3^2e_1 \end{pmatrix}}}

\def\Athree{{\begin{pmatrix} 0 & 0 & 0 \\ 0 & e_1 & 0 \\ 0 & 0 & e_1 \end{pmatrix}}}
\def\Athreea{{\begin{pmatrix} 0 & b & c_1 \\ 0 & 0 & c_2 \\ 0 & 0 & \pm ic_2 \end{pmatrix}}}

\def\Athreeb{{\begin{pmatrix} b_3^2 & b_1 & c \\ 0 & 0 & \pm b_3 \\ 0 & b_3 & 0 \end{pmatrix}}}
\def\AthreebM{{\begin{pmatrix} 0 & 0 & 0 \\ 0 & b_3^2e_1 & 0 \\ 0 & 0 & b_3^2e_1 \end{pmatrix}}}

\def\Athreec{{\begin{pmatrix} b_2^2 & b_1 & c \\ 0 & b_2 & 0 \\ 0 & 0 & \pm b_2 \end{pmatrix}}}
\def\AthreecM{{\begin{pmatrix} 0 & 0 & 0 \\ 0 & b_2^2e_1 & 0 \\ 0 & 0 & b_2^2e_1 \end{pmatrix}}}

\def\Athreed{{\begin{pmatrix} 0 & b_1 & c \\ 0 & b_2 & 0 \\ 0 & \pm ib_2 & 0 \end{pmatrix}}}

\def\Afour{{\begin{pmatrix} 0 & 0 & 0 \\ 0 & 0 & e_1 \\ 0 & e_1 & e_2 \end{pmatrix}}}
\def\Afoura{{\begin{pmatrix} c_3^3 & 2c_2c_3 & c_1 \\ 0 & c_3^2 & c_2 \\ 0 & 0 & c_3 \end{pmatrix}}}
\def\AfouraM{{\begin{pmatrix} 0 & 0 & 0 \\ 0 & 0 & c_3^3e_1 \\ 0 & c_3^3e_1 & 2c_2c_3e_1 + c_3^2e_2 \end{pmatrix}}}

\def\Afive{{\begin{pmatrix} 0 & 0 & 0 \\ 0 & 0 & e_1 \\ 0 & -e_1 & 0 \end{pmatrix}}}
\def\Afivea{{\begin{pmatrix} b_2c_3 - b_3c_2 & b_1 & c_1 \\ 0 & b_2 & c_2 \\ 0 & b_3 & c_3 \end{pmatrix}}}
\def\AfiveaM{{\begin{pmatrix} 0 & 0 & 0 \\ 0 & 0 & (b_2c_3 - b_3c_2)e_1 \\ 0 & (b_3c_2 - b_2c_3)e_1 & 0 \end{pmatrix}}}

\def\Asix{{\begin{pmatrix} 0 & 0 & 0 \\ 0 & e_1 & e_1 \\ 0 & -e_1 & \lambda e_1 \end{pmatrix}}}
\def\Asixa{{\begin{pmatrix} 0 & b_1 & c_1 \\ 0 & 0 & 0 \\ 0 & b_3 & c_3 \end{pmatrix}}}

\def\Asixb{{\begin{pmatrix} b_2^2 & b_1 & c \\ 0 & b_2 & 0 \\ 0 & b_3 & b_2 \end{pmatrix}}}
\def\AsixbM{{\begin{pmatrix} 0 & 0 & 0 \\ 0 & b_2^2e_1 & b_2^2e_1 \\ 0 & -b_2^2e_1 & 0 \end{pmatrix}}}

\def\Asixc{{\begin{pmatrix} 0 & b & c_1 \\ 0 & 0 & \pm i\sqrt{\lambda}c_3 \\ 0 & 0 & c_3\end{pmatrix}}}

\def\Asixd{{\begin{pmatrix} \lambda b_3^2 & b_1 & c \\ 0 & 0 & -\lambda b_3 \\ 0 & b_3 & 0 \end{pmatrix}}}
\def\AsixdM{{\begin{pmatrix} 0 & 0 & 0 \\ 0 & \lambda b_3^2e_1 & \lambda b_3^2e_1 \\ 0 & -\lambda b_3^2e_1 & \lambda^2 b_3^2e_1 \end{pmatrix}}}

\def\Asixe{{\begin{pmatrix} b_2^2 & b_1 & c \\ 0 & b_2 & 0 \\ 0 & 0 & b_2 \end{pmatrix}}}
\def\AsixeM{{\begin{pmatrix} 0 & 0 & 0 \\ 0 & b_2^2e_1 & b_2^2e_1 \\ 0 & -b_2^2e_1 & \lambda b_2^2e_1 \end{pmatrix}}}

\def\Asixf{{\begin{pmatrix} 0 & b_1 & c \\ 0 & \pm i \sqrt{\lambda}b_3 & 0 \\ 0 & b_3 & 0 \end{pmatrix}}}

\def\Aseven{{\begin{pmatrix} 0 & 0 & 0 \\ 0 & 0 & e_1 \\ 0 & \lambda e_1 & e_2 \end{pmatrix}}}
\def\Asevena{{\begin{pmatrix} c_3^3 & (1+\lambda)c_2c_3 & c_1 \\ 0 & c_3^2 & c_2 \\ 0 & 0 & c_3 \end{pmatrix}}}
\def\AsevenaM{{\begin{pmatrix} 0 & 0 & 0 \\ 0 & 0 & c_3^3e_1 \\ 0 & \lambda c_3^3e_1 & (1+\lambda)c_2c_3e_1 + c_3^2e_2  \end{pmatrix}}}

\def\Aeight{{\begin{pmatrix} 0 & 0 & 0 \\ 0 & 0 & 0 \\ 0 & e_1 & e_2 \end{pmatrix}}}
\def\Aeighta{{\begin{pmatrix} c_3^3 & c_2c_3 & c_1 \\ 0 & c_3^2 & c_2 \\ 0 & 0 & c_3 \end{pmatrix}}}
\def\AeightaM{{\begin{pmatrix} 0 & 0 & 0 \\ 0 & 0 & 0 \\ 0 & c_3^3e_1 & c_2c_3e_1 + c_3^2e_2 \end{pmatrix}}}

\def\Anine{{\begin{pmatrix} 0 & 0 & 0 \\ 0 & 0 & 0 \\ 0 & e_2 & 0 \end{pmatrix}}}
\def\Aninea{{\begin{pmatrix} a_1 & 0 & c_1 \\ 0 & 0 & 0 \\ a_3 & 0 & c_3 \end{pmatrix}}}

\def\Anineb{{\begin{pmatrix} a_1 & 0 & c_1 \\ a_2 & 0 & c_2 \\ 0 & 0 & 0 \end{pmatrix}}}

\def\Aninec{{\begin{pmatrix} a & 0 & c \\ 0 & b & 0 \\ 0 & 0 & 1 \end{pmatrix}}}
\def\AninecM{{\begin{pmatrix} 0 & 0 & 0 \\ 0 & 0 & 0 \\ 0 & be_2 & 0 \end{pmatrix}}}

\def\Aten{{\begin{pmatrix} 0 & 0 & 0 \\ 0 & 0 & 0 \\ 0 & e_2 & e_1 \end{pmatrix}}}
\def\Atena{{\begin{pmatrix} 1 & 0 & c \\ 0 & b & 0 \\ 0 & 0 & 1 \end{pmatrix}}}
\def\AtenaM{{\begin{pmatrix} 0 & 0 & 0 \\ 0 & 0 & 0 \\ 0 & be_2 & e_1 \end{pmatrix}}}

\def\Atenb{{\begin{pmatrix} 0 & 0 & c_1 \\ 0 & 0 & c_2 \\ 0 & 0 & 0 \end{pmatrix}}}

\def\Atenc{{\begin{pmatrix} c_3^2 & 0 & c_1 \\ 0 & 0 & 0 \\ 0 & 0 & c_3 \end{pmatrix}}}
\def\AtencM{{\begin{pmatrix} 0 & 0 & 0 \\ 0 & 0 & 0 \\ 0 & 0 & c_3^2e_1 \end{pmatrix}}}

\def\Aeleven{{\begin{pmatrix} 0 & 0 & 0 \\ 0 & 0 & 0 \\ e_1 & \lambda e_2 & 0 \end{pmatrix}}}
\def\Aelevena{{\begin{pmatrix} 0 & 0 & c_1 \\ 0 & 0 & c_2 \\ 0 & 0 & 0 \end{pmatrix}}}

\def\Aelevenb{{\begin{pmatrix} 0 & 0 & 0 \\ 0 & 0 & 0 \\ 0 & 0 & c \end{pmatrix}}}

\def\Aelevenc{{\begin{pmatrix} 0 & 0 & 0 \\ a & 0 & 0 \\ 0 & 0 & \lambda^{-1} \end{pmatrix}}}
\def\AelevencM{{\begin{pmatrix} 0 & 0 & 0 \\ 0 & 0 & 0 \\ ae_2 & 0 & 0 \end{pmatrix}}}

\def\Aelevend{{\begin{pmatrix} 0 & b & 0 \\ 0 & 0 & 0 \\ 0 & 0 & \lambda \end{pmatrix}}}
\def\AelevendM{{\begin{pmatrix} 0 & 0 & 0 \\ 0 & 0 & 0 \\ 0 & \lambda be_1 & 0 \end{pmatrix}}}

\def\Aelevene{{\begin{pmatrix} 0 & b & 0 \\ a & 0 & 0 \\ 0 & 0 & -1 \end{pmatrix}}}
\def\AeleveneM{{\begin{pmatrix} 0 & 0 & 0 \\ 0 & 0 & 0 \\ ae_2 & -be_1 & 0 \end{pmatrix}}}

\def\Aelevenf{{\begin{pmatrix} a_1 & b_1 & 0 \\ a_2 & b_2 & 0 \\ 0 & 0 & 1 \end{pmatrix}}}
\def\AelevenfM{{\begin{pmatrix} 0 & 0 & 0 \\ 0 & 0 & 0 \\ a_1e_1 + a_2e_2 &
 b_1e_1 + b_2e_2 & 0 \end{pmatrix}}}

\def\Aeleveng{{\begin{pmatrix} a & 0 & 0 \\ 0 & b & 0 \\ 0 & 0 & 1 \end{pmatrix}}}
\def\AelevengM{{\begin{pmatrix} 0 & 0 & 0 \\ 0 & 0 & 0 \\ ae_1 & \lambda be_2 & 0 \end{pmatrix}}}

\def\Atwelve{{\begin{pmatrix} 0 & 0 & 0 \\ 0 & 0 & 0 \\ e_1 & e_1+e_2 & 0 \end{pmatrix}}}
\def\Atwelvea{{\begin{pmatrix} 0 & 0 & c_1 \\ 0 & 0 & c_2 \\ 0 & 0 & c_3 \end{pmatrix}}}

\def\Atwelvec{{\begin{pmatrix} a & b & 0 \\ 0 & a & 0 \\ 0 & 0 & 1 \end{pmatrix}}}
\def\AtwelvecM{{\begin{pmatrix} 0 & 0 & 0 \\ 0 & 0 & 0 \\ ae_1 & (a+b)e_1+ae_2 & 0 \end{pmatrix}}}

\def\Athirteen{{\begin{pmatrix} 0 & 0 & 0 \\ 0 & e_1 & 0 \\ e_1 & \frac{1}{2}e_2 & 0 \end{pmatrix}}}
\def\Athirteena{{\begin{pmatrix} 0 & 0 & c_1 \\ 0 & 0 & 0 \\ 0 & 0 & c_3 \end{pmatrix}}}

\def\Athirteenb{{\begin{pmatrix} 0 & b & 0 \\ 0 & 0 & 0 \\ 0 & 0 & \frac{1}{2} \end{pmatrix}}}
\def\AthirteenbM{{\begin{pmatrix} 0 & 0 & 0 \\ 0 & 0 & 0 \\ 0 & \frac{1}{2}be_1 & 0 \end{pmatrix}}}

\def\Athirteenc{{\begin{pmatrix} b^2 & 0 & 0 \\ 0 & b & 0 \\ 0 & 0 & 1 \end{pmatrix}}}
\def\AthirteencM{{\begin{pmatrix} 0 & 0 & 0 \\ 0 & b^2e_1 & 0 \\ b^2e_1 & \frac{1}{2}be_2 & 0 \end{pmatrix}}}

\def\Bzero{{\begin{pmatrix} e_1 & 0 & 0 \\ 0 & e_2 & 0 \\ 0 & 0 & e_3 \end{pmatrix}}}
\def\Bzeroa{{\begin{pmatrix} a_1 & b_1 & c_1 \\ a_2 & b_2 & c_2 \\ a_3 & b_3 & c_3 \end{pmatrix}}}
\def\BzeroaM{{\begin{pmatrix} \sum_{i=1}^3 a_ie_i & 0 & 0 \\ 0 & \sum_{i=1}^3 b_ie_i & 0 \\ 0 & 0 & \sum_{i=1}^3 c_ie_i \end{pmatrix}}}

\def\Bone{{\begin{pmatrix} 0 & 0 & 0 \\ 0 & e_2 & 0 \\ 0 & 0 & e_3 \end{pmatrix}}}
\def\Bonea{{\begin{pmatrix} a & 0 & 0 \\ 0 & 0 & c_2 \\ 0 & 0 & c_3 \end{pmatrix}}}
\def\BoneaM{{\begin{pmatrix} 0 & 0 & 0 \\ 0 & 0 & 0 \\ 0 & 0 & c_2e_2 + c_3e_3 \end{pmatrix}}}

\def\Boneb{{\begin{pmatrix} a & 0 & 0 \\ 0 & b_2 & 0 \\ 0 & b_3 & 0 \end{pmatrix}}}
\def\BonebM{{\begin{pmatrix} 0 & 0 & 0 \\ 0 & b_2e_2 + b_3e_3 & 0 \\ 0 & 0 & 0 \end{pmatrix}}}

\def\Bonec{{\begin{pmatrix} a & 0 & 0 \\ 0 & 0 & c \\ 0 & b & 0 \end{pmatrix}}}
\def\BonecM{{\begin{pmatrix} 0 & 0 & 0 \\ 0 & be_3 & 0 \\ 0 & 0 & ce_2 \end{pmatrix}}}

\def\Boned{{\begin{pmatrix} a & 0 & 0 \\ 0 & b & 0 \\ 0 & 0 & c \end{pmatrix}}}
\def\BonedM{{\begin{pmatrix} 0 & 0 & 0 \\ 0 & be_2 & 0 \\ 0 & 0 & ce_3 \end{pmatrix}}}

\def\Btwo{{\begin{pmatrix} 0 & 0 & e_1 \\ 0 & e_2 & 0 \\ e_1 & 0 & e_3 \end{pmatrix}}}
\def\Btwoa{{\begin{pmatrix} 0 & 0 & 0 \\ 0 & b_2 & c \\ 0 & b_3 & 0 \end{pmatrix}}}
\def\BtwoaM{{\begin{pmatrix} 0 & 0 & 0 \\ 0 & b_2e_2+b_3e_3 & 0 \\ 0 & 0 & ce_2 \end{pmatrix}}}

\def\Btwob{{\begin{pmatrix} a & 0 & 0 \\ 0 & b & c \\ 0 & 0 & 1 \end{pmatrix}}}
\def\BtwobM{{\begin{pmatrix} 0 & 0 & ae_1 \\ 0 & be_2 & 0 \\ ae_1 & 0 & ce_2+e_3 \end{pmatrix}}}

\def\Bthree{{\begin{pmatrix} 0 & 0 & e_1 \\ 0 & e_2 & 0 \\ 0 & 0 & e_3 \end{pmatrix}}}
\def\Bthreea{{\begin{pmatrix} 0 & 0 & 0 \\ 0 & b & c \\ 0 & 0 & 0 \end{pmatrix}}}
\def\BthreeaM{{\begin{pmatrix} 0 & 0 & 0 \\ 0 & be_2 & 0 \\ 0 & 0 & ce_2 \end{pmatrix}}}

\def\Bthreeb{{\begin{pmatrix} 0 & b_1 & 0 \\ 0 & b_2 & c \\ 0 & 1 & 0 \end{pmatrix}}}
\def\BthreebM{{\begin{pmatrix} 0 & 0 & 0 \\ 0 & b_1e_1 + b_2e_2 + e_3 & 0 \\ 0 & 0 & ce_2 \end{pmatrix}}}

\def\Bthreec{{\begin{pmatrix} a & 0 & c_1 \\ 0 & b & c_2 \\ 0 & 0 & 1 \end{pmatrix}}}
\def\BthreecM{{\begin{pmatrix} 0 & 0 & ae_1 \\ 0 & be_2 & 0 \\ 0 & 0 & c_1e_1 + c_2e_2 + e_3 \end{pmatrix}}}

\def\Bfour{{\begin{pmatrix} 0 & 0 & e_1 \\ 0 & e_2 & 0 \\ 0 & 0 & e_1+e_3 \end{pmatrix}}}
\def\Bfoura{{\begin{pmatrix} 0 & 0 & 0 \\ 0 & b & c \\ 0 & 0 & 0 \end{pmatrix}}}
\def\BfouraM{{\begin{pmatrix} 0 & 0 & 0 \\ 0 & be_2 & 0 \\ 0 & 0 & ce_2 \end{pmatrix}}}

\def\Bfourb{{\begin{pmatrix} 1 & 0 & c_1 \\ 0 & b & c_2 \\ 0 & 0 & 1 \end{pmatrix}}}
\def\BfourbM{{\begin{pmatrix} 0 & 0 & e_1 \\ 0 & be_2 & 0 \\ 0 & 0 & (1+c_1)e_1+c_2e_2+e_3 \end{pmatrix}}}

\def\Bfive{{\begin{pmatrix} 0 & 0 & e_1 \\ 0 & e_2 & 0 \\ \lambda e_1 & 0 & e_3 \end{pmatrix}}}
\def\Bfivea{{\begin{pmatrix} 0 & 0 & 0 \\ 0 & 1 & 0 \\ 0 & 0 & 0 \end{pmatrix}}}
\def\BfiveaM{{\begin{pmatrix} 0 & 0 & 0 \\ 0 & e_2 & 0 \\ 0 & 0 & 0 \end{pmatrix}}}

\def\Bfiveb{{\begin{pmatrix} 0 & 0 & 0 \\ 0 & 0 & 1 \\ 0 & 0 & 0 \end{pmatrix}}}
\def\BfivebM{{\begin{pmatrix} 0 & 0 & 0 \\ 0 & 0 & 0 \\ 0 & 0 & e_2 \end{pmatrix}}}

\def\Bfivec{{\begin{pmatrix} 0 & 0 & 0 \\ 0 & 0 & 0 \\ 0 & 1 & 0 \end{pmatrix}}}
\def\BfivecM{{\begin{pmatrix} 0 & 0 & 0 \\ 0 & e_3 & 0 \\ 0 & 0 & 0 \end{pmatrix}}}

\def\Bfived{{\begin{pmatrix} a & 0 & 0 \\ 0 & 0 & 0 \\ 0 & 0 & 1 \end{pmatrix}}}
\def\BfivedM{{\begin{pmatrix} 0 & 0 & ae_1 \\ 0 & 0 & 0 \\ \lambda ae_1 & 0 & e_3 \end{pmatrix}}}

\def\Bfivee{{\begin{pmatrix} 0 & 0 & 0 \\ 0 & 1 & 0 \\ 0 & 1 & 0 \end{pmatrix}}}
\def\BfiveeM{{\begin{pmatrix} 0 & 0 & 0 \\ 0 & e_2+e_3 & 0 \\ 0 & 0 & 0 \end{pmatrix}}}

\def\Bfivef{{\begin{pmatrix} a & 0 & 0 \\ 0 & 1 & 0 \\ 0 & 0 & 1 \end{pmatrix}}}
\def\BfivefM{{\begin{pmatrix} 0 & 0 & ae_1 \\ 0 & e_2 & 0 \\ \lambda ae_1 & 0 & e_3 \end{pmatrix}}}

\def\Bfiveg{{\begin{pmatrix} 0 & 0 & 0 \\ 0 & 0 & 1 \\ 0 & 1 & 0 \end{pmatrix}}}
\def\BfivegM{{\begin{pmatrix} 0 & 0 & 0 \\ 0 & e_3 & 0 \\ 0 & 0 & e_2 \end{pmatrix}}}

\def\Bfiveh{{\begin{pmatrix} a & 0 & 0 \\ 0 & 0 & 1 \\ 0 & 0 & 1 \end{pmatrix}}}
\def\BfivehM{{\begin{pmatrix} 0 & 0 & ae_1 \\ 0 & 0 & 0 \\ \lambda ae_1 & 0 & e_2+e_3 \end{pmatrix}}}

\def\Cone{{\begin{pmatrix} 0 & 0 & 0 \\ 0 & 0 & 0 \\ 0 & 0 & e_3 \end{pmatrix}}}
\def\Conea{{\begin{pmatrix} a_1 & b_1 & 0 \\ a_2 & b_2 & 0 \\ 0 & 0 & c \end{pmatrix}}}
\def\ConeaM{{\begin{pmatrix} 0 & 0 & 0 \\ 0 & 0 & 0 \\ 0 & 0 & ce_3 \end{pmatrix}}}

\def\Ctwo{{\begin{pmatrix} 0 & 0 & e_1 \\ 0 & 0 & 0 \\ e_1 & 0 & e_3 \end{pmatrix}}}
\def\Ctwoa{{\begin{pmatrix} 0 & b_1 & 0 \\ 0 & b_2 & 0 \\ 0 & 0 & 0 \end{pmatrix}}}
\def\Ctwob{{\begin{pmatrix} a & 0 & 0 \\ 0 & b & 0 \\ 0 & 0 & 1 \end{pmatrix}}}
\def\CtwobM{{\begin{pmatrix} 0 & 0 & ae_1 \\ 0 & 0 & 0 \\ ae_1 & 0 & e_3 \end{pmatrix}}}

\def\Cthree{{\begin{pmatrix} 0 & 0 & e_1 \\ 0 & 0 & 0 \\ 0 & 0 & e_3 \end{pmatrix}}}
\def\Cthreea{{\begin{pmatrix} 0 & b_1 & 0 \\ 0 & b_2 & 0 \\ 0 & 0 & 0 \end{pmatrix}}}
\def\Cthreeb{{\begin{pmatrix} a & 0 & c \\ 0 & b & 0 \\ 0 & 0 & 1 \end{pmatrix}}}
\def\CthreebM{{\begin{pmatrix} 0 & 0 & ae_1 \\ 0 & 0 & 0 \\ 0 & 0 & ce_1+e_3 \end{pmatrix}}}

\def\Cfour{{\begin{pmatrix} 0 & 0 & e_1 \\ 0 & 0 & 0 \\ 0 & 0 & e_1+e_3 \end{pmatrix}}}
\def\Cfoura{{\begin{pmatrix} 0 & b_1 & 0 \\ 0 & b_2 & 0 \\ 0 & 0 & 0 \end{pmatrix}}}
\def\Cfourb{{\begin{pmatrix} 1 & 0 & c \\ 0 & b & 0 \\ 0 & 0 & 1 \end{pmatrix}}}
\def\CfourbM{{\begin{pmatrix} 0 & 0 & e_1 \\ 0 & 0 & 0 \\ 0 & 0 & (1+c)e_1+e_3 \end{pmatrix}}}

\def\Cfive{{\begin{pmatrix} 0 & 0 & e_1 \\ 0 & 0 & 0 \\ \lambda e_1 & 0 & e_3 \end{pmatrix}}}
\def\Cfivea{{\begin{pmatrix} 0 & b_1 & 0 \\ 0 & b_2 & 0 \\ 0 & 0 & 0 \end{pmatrix}}}
\def\Cfiveb{{\begin{pmatrix} a & 0 & 0 \\ 0 & b & 0 \\ 0 & 0 & 1 \end{pmatrix}}}
\def\CfivebM{{\begin{pmatrix} 0 & 0 & ae_1 \\ 0 & 0 & 0 \\ \lambda ae_1 & 0 & e_3 \end{pmatrix}}}

\def\Csix{{\begin{pmatrix} 0 & 0 & e_1 \\ 0 & 0 & e_2 \\ e_1 & 0 & e_3 \end{pmatrix}}}
\def\Csixa{{\begin{pmatrix} a & 0 & 0 \\ 0 & b & c \\ 0 & 0 & 1 \end{pmatrix}}}
\def\CsixaM{{\begin{pmatrix} 0 & 0 & ae_1 \\ 0 & 0 & be_2 \\ ae_1 & 0 & ce_2+e_3 \end{pmatrix}}}

\def\Cseven{{\begin{pmatrix} 0 & 0 & e_1 \\ 0 & 0 & e_2 \\ e_1 & 0 & e_2+e_3 \end{pmatrix}}}
\def\Csevena{{\begin{pmatrix} a & 0 & 0 \\ 0 & 1 & c \\ 0 & 0 & 1 \end{pmatrix}}}
\def\CsevenaM{{\begin{pmatrix} 0 & 0 & ae_1 \\ 0 & 0 & e_2 \\ ae_1 & 0 & (1+c)e_2+e_3 \end{pmatrix}}}

\def\Ceight{{\begin{pmatrix} 0 & 0 & e_1 \\ 0 & 0 & e_2 \\ 0 & 0 & e_3 \end{pmatrix}}}
\def\Ceighta{{\begin{pmatrix} a_1 & b_1 & c_1 \\ a_2 & b_2 & c_2 \\ 0 & 0 & 1 \end{pmatrix}}}
\def\CeightaM{{\begin{pmatrix} 0 & 0 & a_1e_1+a_2e_2 \\ 0 & 0 & b_1e_1+b_2e_2 \\ 0 & 0 & c_1e_1+c_2e_2+e_3 \end{pmatrix}}}

\def\Cnine{{\begin{pmatrix} 0 & 0 & e_1 \\ 0 & 0 & e_2 \\ \lambda e_1 & 0 & e_3 \end{pmatrix}}}
\def\Cninea{{\begin{pmatrix} a & 0 & 0 \\ 0 & b & c \\ 0 & 0 & 1 \end{pmatrix}}}
\def\CnineaM{{\begin{pmatrix} 0 & 0 & ae_1 \\ 0 & 0 & be_2 \\ \lambda ae_1 & 0 & ce_2+e_3 \end{pmatrix}}}

\def\Cten{{\begin{pmatrix} 0 & 0 & e_1 \\ 0 & 0 & e_2 \\ \lambda e_1 & 0 & e_2+e_3 \end{pmatrix}}}
\def\Ctena{{\begin{pmatrix} a_1 & 0 & c_1 \\ a_2 & 1 & c_2 \\ 0 & 0 & 1 \end{pmatrix}}}
\def\CtenaM{{\begin{pmatrix} 0 & 0 & a_1e_1+a_2e_2 \\ 0 & 0 & e_2 \\ \lambda(a_1e_1+a_2e_2) & 0 & c_1e_1+(1+c_2)e_2+e_3 \end{pmatrix}}}

\def\Celeven{{\begin{pmatrix} 0 & 0 & e_1 \\ 0 & 0 & e_2 \\ e_1 & e_2 & e_3 \end{pmatrix}}}
\def\Celevena{{\begin{pmatrix} a_1 & b_1 & 0 \\ a_2 & b_2 & 0 \\ 0 & 0 & 1 \end{pmatrix}}}
\def\CelevenaM{{\begin{pmatrix} 0 & 0 & a_1e_1+a_2e_2 \\ 0 & 0 & b_1e_1+b_2e_1 \\ a_1e_1+a_2e_2 & b_1e_1+b_2e_2 & e_3 \end{pmatrix}}}

\def\Ctwelve{{\begin{pmatrix} 0 & 0 & e_1 \\ 0 & 0 & e_2 \\ e_1 & \lambda e_2 & e_3 \end{pmatrix}}}
\def\Ctwelvea{{\begin{pmatrix} a & 0 & 0 \\ 0 & b & 0 \\ 0 & 0 & 1 \end{pmatrix}}}
\def\CtwelveaM{{\begin{pmatrix} 0 & 0 & ae_1 \\ 0 & 0 & be_2 \\ ae_1 & \lambda be_2 & e_3 \end{pmatrix}}}

\def\Cthirteen{{\begin{pmatrix} 0 & 0 & e_1 \\ 0 & 0 & e_2 \\ \lambda e_1 & \eta e_2 & e_3 \end{pmatrix}}}
\def\Cthirteena{{\begin{pmatrix} a & 0 & 0 \\ 0 & b & 0 \\ 0 & 0 & 1 \end{pmatrix}}}
\def\CthirteenaM{{\begin{pmatrix} 0 & 0 & ae_1 \\ 0 & 0 & be_2 \\ \lambda ae_1 & \eta be_2 & e_3 \end{pmatrix}}}
\def\Cthirteenb{{\begin{pmatrix} a_1 & b_1 & 0 \\ a_2 & b_2 & 0 \\ 0 & 0 & 1 \end{pmatrix}}}
\def\CthirteenbM{{\begin{pmatrix} 0 & 0 & a_1e_1+a_2e_2 \\ 0 & 0 & b_1e_1+b_2e_2 \\ \lambda(a_1e_1+a_2e_2) & \eta(b_1e_1+b_2e_2) & e_3 \end{pmatrix}}}

\def\Cfourteen{{\begin{pmatrix} 0 & 0 & e_1 \\ 0 & 0 & e_2 \\ e_1 & e_1+e_2 & e_3 \end{pmatrix}}}
\def\Cfourteena{{\begin{pmatrix} a & b & 0 \\ 0 & a & 0 \\ 0 & 0 & 1 \end{pmatrix}}}
\def\CfourteenaM{{\begin{pmatrix} 0 & 0 & ae_1 \\ 0 & 0 & be_1+ae_2 \\ ae_1 & (a+b)e_1+ae_2 & e_3 \end{pmatrix}}}

\def\Cfifteen{{\begin{pmatrix} 0 & 0 & e_1 \\ 0 & 0 & e_2 \\ \lambda e_1 & e_1+\lambda e_2 & e_3 \end{pmatrix}}}
\def\Cfifteena{{\begin{pmatrix} a & b & 0 \\ 0 & a & 0 \\ 0 & 0 & 1 \end{pmatrix}}}
\def\CfifteenaM{{\begin{pmatrix} 0 & 0 & ae_1 \\ 0 & 0 & be_1+ae_2 \\ \lambda ae_1 & (a+\lambda b)e_1+\lambda ae_2 & e_3 \end{pmatrix}}}

\def\Csixteen{{\begin{pmatrix} 0 & 0 & e_1 \\ 0 & 0 & e_2 \\ 0 & e_1 & e_3 \end{pmatrix}}}
\def\Csixteena{{\begin{pmatrix} a & b & c \\ 0 & a & 0 \\ 0 & 0 & 1 \end{pmatrix}}}
\def\CsixteenaM{{\begin{pmatrix} 0 & 0 & ae_1 \\ 0 & 0 & be_1+ae_2 \\ 0 & ae_1 & ce_1+e_3 \end{pmatrix}}}

\def\Cseventeen{{\begin{pmatrix} 0 & 0 & e_1 \\ 0 & 0 & e_2 \\ 0 & e_1 & e_2+e_3 \end{pmatrix}}}
\def\Cseventeena{{\begin{pmatrix} 1 & b & c \\ 0 & 1 & b \\ 0 & 0 & 1 \end{pmatrix}}}
\def\CseventeenaM{{\begin{pmatrix} 0 & 0 & e_1 \\ 0 & 0 & be_1+e_2 \\ 0 & e_1 & (b+c)e_1+(1+b)e_2+e_3 \end{pmatrix}}}

\def\Ceighteen{{\begin{pmatrix} 0 & 0 & e_1+e_2 \\ 0 & 0 & e_2 \\ 0 & -e_2 & e_3 \end{pmatrix}}}
\def\Ceighteena{{\begin{pmatrix} a & 0 & c \\ 0 & a & c \\ 0 & 0 & 1 \end{pmatrix}}}
\def\CeighteenaM{{\begin{pmatrix} 0 & 0 & a(e_1+e_2) \\ 0 & 0 & ae_2 \\ 0 & -ae_2 & c(e_1+e_2)+e_3 \end{pmatrix}}}

\def\Cnineteen{{\begin{pmatrix} 0 & 0 & e_1+e_2 \\ 0 & 0 & e_2 \\ 0 & -e_2 & e_1+e_3 \end{pmatrix}}}
\def\Cnineteena{{\begin{pmatrix} 1 & 0 & c \\ 0 & 1 & c \\ 0 & 0 & 1 \end{pmatrix}}}
\def\CnineteenaM{{\begin{pmatrix} 0 & 0 & e_1+e_2 \\ 0 & 0 & e_2 \\ 0 & -e_2 & (1+c)e_1+ce_2+e_3 \end{pmatrix}}}

\def\Done{{\begin{pmatrix} e_2 & 0 & 0 \\ 0 & 0 & 0 \\ 0 & 0 & e_3 \end{pmatrix}}}
\def\Donea{{\begin{pmatrix} a_1 & 0 & 0 \\ a_2 & a_1^2 & 0 \\ 0 & 0 & c \end{pmatrix}}}
\def\DoneaM{{\begin{pmatrix} a_1^2e_2 & 0 & 0 \\ 0 & 0 & 0 \\ 0 & 0 & ce_3 \end{pmatrix}}}

\def\Dtwo{{\begin{pmatrix} e_2 & 0 & e_1 \\ 0 & 0 & e_2 \\ e_1 & e_2 & e_3 \end{pmatrix}}}
\def\Dtwoa{{\begin{pmatrix} a_1 & 0 & 0 \\ a_2 & a_1^2 & 0 \\ 0 & 0 & 1 \end{pmatrix}}}
\def\DtwoaM{{\begin{pmatrix} a_1^2e_2 & 0 & a_1e_1+a_2e_2 \\ 0 & 0 & a_1^2e_2 \\ a_1e_1+a_2e_2 & a_1^2e_2 & e_3 \end{pmatrix}}}

\def\Dthree{{\begin{pmatrix} e_2 & 0 & e_1 \\ 0 & 0 & e_2 \\ e_1+e_2 & e_2 & e_3 \end{pmatrix}}}
\def\Dthreea{{\begin{pmatrix} a_1 & 0 & 0 \\ a_2 & a_1 & 0 \\ 0 & 0 & 1 \end{pmatrix}}}
\def\DthreeaM{{\begin{pmatrix} a_1e_2 & 0 & a_1e_1+a_2e_2 \\ 0 & 0 & a_1e_2 \\ a_1e_1+(a_1+a_2)e_2 & a_1e_2 & e_3 \end{pmatrix}}}

\def\Dfour{{\begin{pmatrix} e_2 & 0 & e_1 \\ 0 & 0 & e_2 \\ \frac{1}{2}e_1 & 0 & e_3 \end{pmatrix}}}
\def\Dfoura{{\begin{pmatrix} a & 0 & 0 \\ 0 & a^2 & c \\ 0 & 0 & 1 \end{pmatrix}}}
\def\DfouraM{{\begin{pmatrix} a^2e_2 & 0 & ae_1 \\ 0 & 0 & a^2e_2 \\ \frac{1}{2}ae_1 & 0 & ce_2+e_3 \end{pmatrix}}}

\def\Dfive{{\begin{pmatrix} e_2 & 0 & e_1 \\ 0 & 0 & e_2 \\ \frac{1}{2}e_1 & 0 & e_2+e_3 \end{pmatrix}}}
\def\Dfivea{{\begin{pmatrix} \pm 1 & 0 & 0 \\ 0 & 1 & c \\ 0 & 0 & 1 \end{pmatrix}}}
\def\DfiveaM{{\begin{pmatrix} e_2 & 0 & \pm e_1 \\ 0 & 0 & e_2 \\ \pm\frac{1}{2}e_1 & 0 & (1+c)e_2+e_3 \end{pmatrix}}}

\def\Dsix{{\begin{pmatrix} e_2 & 0 & e_1 \\ 0 & 0 & e_2 \\ \lambda e_1 & (2\lambda - 1)e_2 & e_3 \end{pmatrix}}}
\def\Dsixa{{\begin{pmatrix} a & 0 & 0 \\ 0 & a^2 & 0 \\ 0 & 0 & 1 \end{pmatrix}}}
\def\DsixaM{{\begin{pmatrix} a^2e_2 & 0 & ae_1 \\ 0 & 0 & a^2e_2 \\ \lambda ae_1 & (2\lambda - 1)a^2e_2 & e_3 \end{pmatrix}}}
\def\Dsixb{{\begin{pmatrix} 0 & 0 & c \\ 0 & 0 & c^2 \\ 0 & 0 & 1 \end{pmatrix}}}
\def\DsixbM{{\begin{pmatrix} 0 & 0 & 0 \\ 0 & 0 & 0 \\ 0 & 0 & ce_1+c^2e_2+e_3 \end{pmatrix}}}

\def\Eone{{\begin{pmatrix} 0 & 0 & 0 \\ -e_1 & 0 & 0 \\ 0 & 0 & e_3 \end{pmatrix}}}
\def\Eonea{{\begin{pmatrix} 0 & b & 0 \\ 0 & 0 & 0 \\ 0 & 0 & c \end{pmatrix}}}
\def\EoneaM{{\begin{pmatrix} 0 & 0 & 0 \\ 0 & 0 & 0 \\ 0 & 0 & ce_3 \end{pmatrix}}}
\def\Eoneb{{\begin{pmatrix} a & 0 & 0 \\ 0 & 1 & 0 \\ 0 & 0 & c \end{pmatrix}}}
\def\EonebM{{\begin{pmatrix} 0 & 0 & 0 \\ -ae_1 & 0 & 0 \\ 0 & 0 & ce_3 \end{pmatrix}}}
\def\Eonec{{\begin{pmatrix} 0 & 0 & 0 \\ 0 & b & 0 \\ 0 & 0 & c \end{pmatrix}}}
\def\EonecM{{\begin{pmatrix} 0 & 0 & 0 \\ 0 & 0 & 0 \\ 0 & 0 & ce_3 \end{pmatrix}}}

\begin{document}

\title{Hom-Novikov algebras}
\author{Donald Yau}

\begin{abstract}
We study a twisted generalization of Novikov algebras, called Hom-Novikov algebras, in which the two defining identities are twisted by a linear map.  It is shown that Hom-Novikov algebras can be obtained from Novikov algebras by twisting along any algebra endomorphism.  All algebra endomorphisms on complex Novikov algebras of dimensions two or three are computed, and their associated Hom-Novikov algebras are described explicitly.  Another class of Hom-Novikov algebras is constructed from Hom-commutative algebras together with a derivation, generalizing a construction due to Dorfman and Gel'fand.  Two other classes of Hom-Novikov algebras are constructed from Hom-Lie algebras together with a suitable linear endomorphism, generalizing a construction due to Bai and Meng.
\end{abstract}



\address{Department of Mathematics\\
    The Ohio State University at Newark\\
    1179 University Drive\\
    Newark, OH 43055, USA}
\email{dyau@math.ohio-state.edu}

\date{\today}
\maketitle

\sqsp

\section{Introduction and Main Results}

Novikov algebras were introduced in the studies of Hamiltonian operators and Poisson brackets of hydrodynamic type \cite{bn,dn1,dn2,dg1,dg2,dg3,osborn1,osborn2,osborn3,xu1,xu2,xu3}.  They are closely related to many topics in mathematical physics and geometry, including Lie groups \cite{bm2,bm6,burde2,perea}, Lie algebras \cite{bm1,bm7,bdv,zelmanov}, affine manifolds \cite{kim}, convex homogeneous cones \cite{vinberg}, rooted tree algebras \cite{cayley}, vector fields \cite{burde1}, and vertex and conformal algebras \cite{borcherds,kac}.  Novikov algebras form a subclass of the class of left-symmetric algebras \cite{burde1}.  In particular, they are Lie-admissible algebras, which are important in some physical applications, such as quantum mechanics and hadronic structures \cite{albert,myung1,myung2,mos,santilli}.  In other words, if $A$ is a left-symmetric algebra (such as a Novikov algebra), then $A$ gives rise to a Lie algebra whose Lie bracket is the commutator bracket, $[x,y] = xy - yx$ for $x,y \in A$.

To be more precise, recall that a \emph{left-symmetric algebra}, also called a Vinberg algebra or a left pre-Lie algebra, is a vector space $A$ (over a ground field $\bk$ of characteristic $0$) together with a bilinear multiplication $\mu \colon A^{\otimes 2} \to A$ such that
\begin{equation}
\label{eq:leftsymmetric}
(xy)z - x(yz) = (yx)z - y(xz)
\end{equation}
for $x,y,z \in A$.  Here, and in the sequel, we often write $\mu(x,y)$ as $xy$.  In other words, if $a(x,y,z) = (xy)z - x(yz)$ denotes the associator, then \eqref{eq:leftsymmetric} says that $a(x,y,z)$ is symmetric in the first two variables, hence the name left-symmetric.  A \emph{Novikov algebra} is a left-symmetric algebra $A$ that satisfies the additional property
\begin{equation}
\label{eq:rightcom}
(xy)z = (xz)y
\end{equation}
for $x,y,z \in A$.  In other words, if $R_y$ denotes the right multiplication operator $x \mapsto xy$, then \eqref{eq:rightcom} says that the right multiplication operators on $A$ commute with one another.  Classifications of Novikov algebras, possibly with additional properties, are known in low dimensions \cite{bm1,bm3,bm3.5,bm4,bm5,bm7,bmh,bmh2,cz,ck,zc}.

The purpose of this paper is to study a twisted version of Novikov algebras, called \emph{Hom-Novikov algebras}, which are motivated by recent work related to Hom-type algebras.  A \emph{Hom-Lie algebra} consists of a vector space $L$, a linear self-map $\alpha$, and a bilinear, skew-symmetric bracket $[-,-] \colon L^{\otimes 2} \to L$, satisfying (i) $\alpha[x,y] = [\alpha(x),\alpha(y)]$ (multiplicativity) and (ii) the following \emph{Hom-Jacobi identity} for $x,y,z \in L$:
\begin{equation}
\label{eq:HomJacobi}
[[x,y],\alpha(z)] + [[z,x],\alpha(y)] + [[y,z],\alpha(x)] = 0.
\end{equation}
Lie algebras are examples of Hom-Lie algebras in which $\alpha$ is the identity map.  Earlier precursors of Hom-Lie algebras can be found in \cite{hu,liu}.  Hom-Lie algebras were introduced in \cite{hls} (without multiplicativity) to describe the structure on some $q$-deformations of the Witt and the Virasoro algebras.  Hom-Lie algebras are closely related to discrete and deformed vector fields and differential calculus \cite{hls,ls,ls2} and have applications to number theory \cite{larsson}, generalizations of the various Yang-Baxter equations, braid group representations, and related algebraic objects \cite{yau5,yau6,yau7,yau8,yau9}.

Just as Lie algebras are closely related to associative algebras, Hom-Lie algebras are closely related to the so-called Hom-associative algebras.  A \emph{Hom-associative algebra} \cite{ms} consists of a vector space $A$, a linear self-map $\alpha$, and a bilinear map $\mu \colon A^{\otimes 2} \to A$, satisfying (i) $\alpha(\mu(x,y)) = \mu(\alpha(x),\alpha(y))$ (multiplicativity) and (ii) Hom-associativity,
\begin{equation}
\label{eq:HomAs}
(xy)\alpha(z) = \alpha(x)(yz)
\end{equation}
for $x,y,z \in A$.  It is shown in \cite{ms} that, if $(A,\mu,\alpha)$ is a Hom-associative algebra, then $(A,[-,-],\alpha)$ is a Hom-Lie algebra, where $[x,y] = xy - yx$ is the commutator bracket.  Conversely, Hom-Lie algebras have universal enveloping Hom-associative algebras \cite{yau,yau3}.  Many examples of Hom-Lie and Hom-associative algebras are given in \cite{ms,ms4,yau2}.  Free Hom-algebras can be obtained using the constructions in \cite{yau,yau3}.  Related Hom-type algebras, in which the defining identities are similarly twisted by a linear map, have been studied in \cite{am,ams,cg,fg,fg2,gohr,makhlouf,ms2,ms3,yau3,yau4}.

Following the patterns of Hom-Lie and Hom-associative algebras, we define a \emph{Hom-Novikov algebra} as a triple $(A,\mu,\alpha)$ in which $A$ is a vector space, $\mu \colon A^{\otimes 2} \to A$ is a bilinear map, and $\alpha \colon A \to A$ is a linear map, satisfying the following three conditions for $x,y,z \in A$:
\begin{subequations}
\label{eq:HomN}
\begin{align}
\alpha(xy) &= \alpha(x)\alpha(y), \label{eq:HomN0}\\
(xy)\alpha(z) - \alpha(x)(yz) &= (yx)\alpha(z) - \alpha(y)(xz), \label{eq:HomN1}\\
(xy)\alpha(z) &= (xz)\alpha(y).\label{eq:HomN2}
\end{align}
\end{subequations}
Comparing \eqref{eq:HomN} with \eqref{eq:leftsymmetric} and \eqref{eq:rightcom}, we see that Novikov algebras are examples of Hom-Novikov algebras in which $\alpha$ is the identity map.  For a Hom-Novikov algebra $(A,\mu,\alpha)$, we call $\mu$ the \emph{Hom-Novikov product} of $A$.  Likewise, if only \eqref{eq:HomN0} and \eqref{eq:HomN1} are satisfied, then we call $(A,\mu,\alpha)$ a \emph{Hom-left-symmetric algebra}.  In particular, a Hom-Novikov algebra is a Hom-left-symmetric algebra that also satisfies \eqref{eq:HomN2}.

We now construct several classes of Hom-Novikov algebras, starting with either a Novikov algebra or a certain kind of Hom-algebra together with a suitable self-map.  Proofs will be given in later sections.

First, from the defining axioms \eqref{eq:HomN}, one can think of a Hom-Novikov algebra as an $\alpha$-twisted version of a Novikov algebra.  This intuitive interpretation can be made precise.  There is a general strategy, first used in \cite[Theorem 2.3]{yau2} and later in \cite{am,ams,fg,fg2,gohr,makhlouf,ms4,yau3,yau4}, to deform an algebraic structure to the corresponding type of Hom-algebra via an endomorphism.  If $\mu \colon A^{\otimes 2} \to A$ is a bilinear multiplication, then an \emph{algebra morphism} $f \colon A \to A$ is a linear map that commutes with the multiplication in $A$, i.e., $f \circ \mu = \mu \circ f^{\otimes 2}$.  The following result says that Novikov algebras deform into Hom-Novikov algebras along any algebra morphism.  This yields a large class of examples of Hom-Novikov algebras.

\begin{theorem}
\label{thm:deform}
Let $(A,\mu)$ be a Novikov algebra and $\alpha \colon A \to A$ be an algebra morphism.  Then $A_\alpha = (A,\mu_\alpha  = \alpha \circ \mu,\alpha)$ is a Hom-Novikov algebra.
\end{theorem}

We call $(A,\mu_\alpha,\alpha)$ a Hom-Novikov deformation of the Novikov algebra $(A,\mu)$ along the algebra morphism $\alpha$.  One can use Theorem \ref{thm:deform} to obtain many examples of Hom-Novikov algebras.  In order to apply Theorem \ref{thm:deform} on a specific Novikov algebra $A$, one has to know at least some of the algebra morphisms on $A$.  As an illustration of the utility of Theorem \ref{thm:deform}, in Sections \ref{sec:2d} and \ref{sec:3d} we will classify all the algebra morphisms $\alpha$ on all the complex $2$-dimensional and $3$-dimensional Novikov algebras, and describe their corresponding Hom-Novikov products $\mu_\alpha = \alpha \circ \mu$.  We will use the classification of Novikov algebras over $\bC$ in dimensions at most three given in \cite{bm3}.  As a result, we obtain all the complex $2$-dimensional and $3$-dimensional Hom-Novikov algebras that can possibly be constructed using the twisting method in Theorem \ref{thm:deform}.  The sub-classes of algebra \emph{automorphisms} on complex Novikov algebras of dimensions $2$ and $3$ were computed in \cite{bm8}.

Next we discuss constructions of Hom-Novikov algebras that mimic known constructions of Novikov algebras.  One such construction of Novikov algebras, due to Dorfman and Gel'fand \cite{dg3}, starts with an associative and commutative algebra $(A,\mu)$ and a derivation $D \colon A \to A$.  The new product
\begin{equation}
\label{eq:GD}
a \ast b = \mu(a,D(b)) = aD(b)
\end{equation}
for $a, b \in A$ makes $(A,\ast)$ into a Novikov algebra.  To generalize this construction, we define a \emph{Hom-commutative algebra} to be a Hom-associative algebra whose multiplication is commutative.  A \emph{derivation} on a Hom-associative algebra is defined in the usual way.  Then we have the following result, generalizing the Dorfman-Gel'fand product \eqref{eq:GD}.

\begin{theorem}
\label{thm:GD}
Let $(A,\mu,\alpha)$ be a Hom-commutative algebra and $D \colon A \to A$ be a derivation such that $D\alpha = \alpha D$.  Then $(A,\ast,\alpha)$ is a Hom-Novikov algebra, where $\ast$ is defined as in \eqref{eq:GD}.
\end{theorem}

The following result is a consequence of Theorem \ref{thm:GD}.

\begin{corollary}
\label{cor:GD}
Let $(A,\mu)$ be an associative and commutative algebra, $\alpha \colon A \to A$ be an algebra morphism, and $D \colon A \to A$ be a derivation such that $D \alpha = \alpha D$.  Then $(A,\ast,\alpha)$ is a Hom-Novikov algebra, where
\begin{equation}
\label{eq:astalpha}
x \ast y = \alpha(xD(y))
\end{equation}
for $x,y \in A$.
\end{corollary}

Examples that illustrate Corollary \ref{cor:GD} will be given in Section \ref{sec:proof}.

Another known construction of Novikov algebras, due to Bai and Meng \cite{bm1}, starts with a Lie algebra and a suitable linear self-map.  It is a Lie algebra analogue of the Dorfman-Gel'fand product \eqref{eq:GD}.  For a Hom-Lie algebra $(L,[-,-],\alpha)$, write $Z(\alpha(L))$ for the subset of $L$ consisting of elements $x$ such that $[x,\alpha(y)] = 0$ for all $y \in L$.  Clearly if $L$ is a Lie algebra (i.e. $\alpha = Id$), then $Z(Id(L)) = Z(L)$ is the center of $L$.

\begin{theorem}
\label{thm:BM}
Let $(L,[-,-],\alpha)$ be a Hom-Lie algebra and $f \colon L \to L$ be a linear map such that $f\alpha = \alpha f$.  Define the products
\begin{equation}
\label{eq:BMproduct}
x \star y = [f(x),y], \quad x \star' y = [x,f(y)]
\end{equation}
for $x,y \in L$.  Then we have:
\begin{enumerate}
\item
$(L,\star,\alpha)$ is a Hom-Novikov algebra if and only if the following conditions hold for $x,y,z \in L$:
\begin{equation}
\label{eq:BM1}
f([f(x),y] + [x,f(y)]) - [f(x),f(y)] \in Z(\alpha(L))
\end{equation}
and
\begin{equation}
\label{eq:BM2}
[f([f(x),y]),\alpha(z)] = [f([f(x),z]),\alpha(y)].
\end{equation}
\item
$(L,\star',\alpha)$ is a Hom-Novikov algebra if and only if the following conditions hold for $x,y,z \in L$:
\begin{equation}
\label{eq:BM1'}
[[x,f(y)] + [f(x),y], f(\alpha(z))] - [\alpha(x),f([y,f(z)])] + [\alpha(y),f([x,f(z)])] = 0
\end{equation}
and
\begin{equation}
\label{eq:BM2'}
[f(x),f(y)] \in Z(\alpha(L)).
\end{equation}
\end{enumerate}
\end{theorem}

Theorem \ref{thm:BM} is a generalization of a result in \cite{bm1} corresponding to $\alpha = Id$.  One can use Theorem \ref{thm:BM} to generate Hom-Novikov algebras, starting with any Hom-Lie algebra $L = (L,[-,-],\alpha)$ and any linear map $f$ on $L$ such that $f \alpha = \alpha f$.  Indeed, there is a unique (up to isomorphism) largest quotient Hom-Lie algebra $L' = (L',[-,-]',\alpha')$ of $L$ on which $f$ has an induced map $f'$ and in which the conditions \eqref{eq:BM1} and \eqref{eq:BM2} hold.  In particular, $(L',\star ,\alpha')$ is a Hom-Novikov algebra, where $\star$ is defined as in \eqref{eq:BMproduct} using $[-,-]'$ and $f'$.  The quotient Hom-Lie algebra $L'$ can be constructed using an inductive procedure similar to the one in \cite[4.2]{yau} and \cite[2.8]{yau3}.  There is a similar construction adapted to the other product $\star'$ in \eqref{eq:BMproduct} and the conditions \eqref{eq:BM1'} and \eqref{eq:BM2'}.  In general, however, the structure of the quotient $L'$ is not well understood.

Note that the expression in \eqref{eq:BM1} is interesting from another viewpoint.  Indeed, a Rota-Baxter operator \cite{baxter,rota1,rota2,rota3} $\beta$ on an associative algebra is a linear map that satisfies $\beta(\beta(x)y + x\beta(y)) = \beta(x)\beta(y)$.  The expression in \eqref{eq:BM1} is closely related to the Lie algebra version of a Rota-Baxter operator \cite{rsts,sts} and the operator form of the classical Yang-Baxter equation \cite{bai,ft}.

The rest of this paper is organized as follows.  In Section \ref{sec:proof}, we will prove the results stated above.  At the end of Section \ref{sec:proof}, we will give examples that illustrate Corollary \ref{cor:GD}, involving (Laurent) polynomial algebras (Example \ref{ex:poly}) and nilpotent derivations (Example \ref{ex:der}).  In Section \ref{sec:2d}, we will classify all the algebra morphisms $\alpha$ on all the complex $2$-dimensional Novikov algebras and describe their associated Hom-Novikov products $\mu_\alpha = \alpha \circ \mu$ (Theorem ~\ref{thm:deform}).  In Section \ref{sec:3d}, we will do the same thing for $3$-dimensional Novikov algebras and make a few concluding remarks.

\section{Proofs}
\label{sec:proof}

\begin{lemma}
\label{lem:mu}
Let $A = (A,\mu)$ be a not-necessarily associative algebra and $\alpha \colon A \to A$ be an algebra morphism.  Then the multiplication $\mu_\alpha = \alpha \circ \mu$ satisfies
\begin{equation}
\label{eq:mualpha}
\mu_\alpha(\mu_\alpha(x,y),\alpha(z)) = \alpha^2((xy)z) \quad\text{and}\quad
\mu_\alpha(\alpha(x),\mu_\alpha(y,z)) = \alpha^2(x(yz))
\end{equation}
for $x,y,z \in A$, where $\alpha^2 = \alpha \circ \alpha$.  Moreover, $\alpha$ is multiplicative with respect to $\mu_\alpha$, i.e., $\alpha \circ \mu_\alpha = \mu_\alpha \circ \alpha^{\otimes 2}$.
\end{lemma}

\begin{proof}
Using the hypothesis that $\alpha$ is an algebra morphism, we have
\[
\begin{split}
\mu_\alpha(\mu_\alpha(x,y),\alpha(z))
&= \alpha(\alpha(xy)\alpha(z))\\
&= \alpha^2((xy)z),
\end{split}
\]
proving the first assertion in \eqref{eq:mualpha}.  The other assertion in \eqref{eq:mualpha} is proved similarly.  For the last assertion, observe that both $\alpha \circ \mu_\alpha$ and $\mu_\alpha \circ \alpha^{\otimes 2}$ are equal to $\alpha \circ \mu \circ \alpha^{\otimes 2}$.
\end{proof}

\begin{proof}[Proof of Theorem \ref{thm:deform}]
It is observed in \cite[Corollary 2.5 (3)]{yau2} that, if $A$ is a left-symmetric algebra, then $A_\alpha$ is a Hom-left-symmetric algebra.  Therefore, it remains to show \eqref{eq:HomN2} for the multiplication $\mu_\alpha = \alpha \circ \mu$.  We compute as follows:
\[
\begin{split}
\mu_\alpha(\mu_\alpha(x,y),\alpha(z))
&= \alpha^2((xy)z) \quad \text{by Lemma ~\ref{lem:mu}}\\
&= \alpha^2((xz)y) \quad \text{by \eqref{eq:rightcom}} \\
&= \mu_\alpha(\mu_\alpha(x,z),\alpha(y))  \quad \text{by Lemma ~\ref{lem:mu}}.
\end{split}
\]
This proves \eqref{eq:HomN2} for the multiplication $\mu_\alpha$.
\end{proof}

\begin{proof}[Proof of Theorem \ref{thm:GD}]
The multiplicativity of $\alpha$ with respect to $\ast$ \eqref{eq:GD} follows from the multiplicativity of $\alpha$ with respect to $\mu$ and the hypothesis $D\alpha = \alpha D$.  Next we check \eqref{eq:HomN1} for the multiplication $\ast$.  We have
\[
\begin{split}
(x \ast y) \ast \alpha(z) - \alpha(x) \ast (y \ast z)
&= (xD(y))D(\alpha(z)) - \alpha(x)D(yD(z))\\
&= (xD(y))\alpha(D(z)) - \alpha(x)(D(y)D(z)) - \alpha(x)(yD^2(z))\\
&= -(xy)\alpha(D^2(z)).
\end{split}
\]
The last equality follows from Hom-associativity \eqref{eq:HomAs}.  Since $\mu$ is assumed to be commutative, the expression $-(xy)\alpha(D^2(z))$ is symmetric in $x$ and $y$, proving \eqref{eq:HomN1}.

For \eqref{eq:HomN2}, we have
\[
\begin{split}
(x \ast y) \ast \alpha(z) &= (xD(y))\alpha(D(z))\\
&= \alpha(x)(D(y)D(z)),
\end{split}
\]
which is symmetric in $y$ and $z$ because $\mu$ is commutative.
\end{proof}

\begin{proof}[Proof of Corollary \ref{cor:GD}]
By \cite[Corollary 2.5 (1)]{yau2}, $(A,\mu_\alpha = \alpha \circ \mu,\alpha)$ is a Hom-associative algebra.  Since $\mu$ is commutative, so is $\mu_\alpha$.  Thus, using Theorem \ref{thm:GD}, it suffices to show that $D$ is a derivation with respect to the multiplication $\mu_\alpha$.  We compute as follows:
\[
\begin{split}
D(\mu_\alpha(x,y)) &= D(\alpha(xy))\\
&= \alpha(D(xy))\\
&= \alpha(D(x)y + xD(y))\\
&= \mu_\alpha(D(x),y) + \mu_\alpha(x,D(y)).
\end{split}
\]
This shows that $D$ is a derivation with respect to the multiplication $\mu_\alpha$, as desired.
\end{proof}

\begin{proof}[Proof of Theorem \ref{thm:BM}]
First note that $\alpha$ is multiplicative with respect to both $\star$ and $\star'$ \eqref{eq:BMproduct}, since $\alpha \circ [-,-] = [-,-] \circ \alpha^{\otimes 2}$ and $f\alpha = \alpha f$.

Consider the first assertion.  The condition ~\eqref{eq:HomN1} for the multiplication $\star$ \eqref{eq:BMproduct} says
\[
[f([f(x),y]),\alpha(z)] - [f(\alpha(x)),[f(y),z]] = [f([f(y),x]),\alpha(z)] - [f(\alpha(y)),[f(x),z]].
\]
Using $f\alpha = \alpha f$ and the skew-symmetry of $[-,-]$, we can rewrite the above equality as
\[
\begin{split}
[f([f(x),y]) + f([x,f(y)]),\alpha(z)]
&= [\alpha(f(x)),[f(y),z]] + [\alpha(f(y)),[z,f(x)]]\\
&= -[\alpha(z),[f(x),f(y)]].
\end{split}
\]
The last equality follows from the Hom-Jacobi identity \eqref{eq:HomJacobi}.  Thus, ~\eqref{eq:HomN1} holds for the multiplication $\star$ if and only if
\[
[f([f(x),y]) + f([x,f(y)]) - [f(x),f(y)],\alpha(z)] = 0,
\]
which is equivalent to \eqref{eq:BM1}.  The condition \eqref{eq:BM2} is simply a restatement of \eqref{eq:HomN2} for the multiplication $\star$.  This proves the first assertion of Theorem ~\ref{thm:BM}.

The second assertion is proved by essentially the same argument, with \eqref{eq:BM1'} and \eqref{eq:BM2'} corresponding to \eqref{eq:HomN1} and \eqref{eq:HomN2}, respectively.  Indeed, \eqref{eq:BM1'} is a restatement of \eqref{eq:HomN1} for the multiplication $\star'$ \eqref{eq:BMproduct}.  On the other hand, the condition \eqref{eq:HomN2} for the multiplication $\star'$ says
\[
[[x,f(y)],f(\alpha(z))] = [[x,f(z)],f(\alpha(y))],
\]
which is equivalent to
\[
[[x,f(y)],\alpha(f(z))] + [[f(z),x],\alpha(f(y))] = 0.
\]
Using the Hom-Jacobi identity \eqref{eq:HomJacobi}, the previous line is equivalent to
\[
[[f(y),f(z)],\alpha(x)] = 0,
\]
which is exactly \eqref{eq:BM2'}.
\end{proof}

The following two examples illustrate how Corollary \ref{cor:GD} can be applied to create concrete examples of Hom-Novikov algebras.

\begin{example}[\textbf{Hom-Novikov algebras from (Laurent) polynomial algebras}]
\label{ex:poly}
Consider the one-variable polynomial algebra $\bk[x]$, the differential operator $D = d/dx$ on $\bk[x]$, and the algebra morphism $\alpha$ on $\bk[x]$ determined by $\alpha(x) = x+c$, where $c \in \bk$ is a fixed element.  We have
\[
\begin{split}
D(\alpha(x^n)) &= D((x+c)^n)\\
&= n(x+c)^{n-1}\\
&= \alpha(D(x^n)),
\end{split}
\]
which is sufficient to conclude that $D\alpha = \alpha D$.  By Corollary \ref{cor:GD}, the new product
\[
f \ast g = \alpha(fD(g)) = f(x+c)(D(g)(x+c))
\]
yields a Hom-Novikov algebra $(\bk[x],\ast,\alpha)$.

The previous paragraph can be easily generalized to the $n$-variable (Laurent) case.  Essentially the same reasoning as above gives the Hom-Novikov algebras $(\bk[x_1, \ldots , x_n],\ast,\alpha)$ and $(\bk[x_1^{\pm 1}, \ldots , x_n^{\pm 1}],\ast,\alpha)$, where $\alpha$ is determined by $\alpha(x_j) = x_j + c_j$ with each $c_j \in \bk$.  The Hom-Novikov product $f \ast g$ is again given by $\alpha(fD(g))$, where $D = \partial/\partial x_i$ is the partial differential operator with respect to $x_i$ for some fixed $i \in \{1, \ldots , n\}$.\qed
\end{example}

\begin{example}[\textbf{Hom-Novikov algebras from nilpotent derivations}]
\label{ex:der}
Let $A$ be an associative and commutative algebra and $D \colon A \to A$ be a \emph{nilpotent} derivation on $A$.  In other words, $D$ is a derivation such that $D^n = 0$ for some $n \geq 2$.  For example, if $x \in A$ is a nilpotent element, say, $x^n = 0$, then the linear self-map $ad(x)$ on $A$ defined by $ad(x)(y) = xy-yx$ is a nilpotent derivation on $A$.  Given such a nilpotent derivation $D$ with $D^n = 0$, the formal exponential map
\[
E = \exp D = Id_A + D + \frac{1}{2!}D^2 + \cdots + \frac{1}{(n-1)!}D^{n-1}
\]
is an algebra automorphism on $A$ \cite[p.26]{abe}.  Since $E$ is a polynomial in $D$, it follows that $DE = ED$.  Therefore, by Corollary ~\ref{cor:GD}, the new product $x \ast y = E(xD(y))$ yields a Hom-Novikov algebra $(A,\ast,E)$.\qed
\end{example}

\section{Examples of $2$-dimensional Hom-Novikov algebras}
\label{sec:2d}

In this and the next sections, we work over the ground field $\bC$ of complex numbers.  The purpose of this section is to classify all the algebra morphisms $\alpha$ on all the $2$-dimensional Novikov algebras $(A,\mu)$ (i.e., $\alpha \circ \mu = \mu \circ \alpha^{\otimes 2}$), using the classification of $2$-dimensional Novikov algebras in \cite{bm3}.  From Theorem ~\ref{thm:deform}, we then obtain their corresponding $2$-dimensional Hom-Novikov algebras $A_\alpha = (A,\mu_\alpha = \alpha \circ \mu,\alpha)$.

First let us establish some notations, which will also be used in the next section.  Let $(A,\mu)$ be an $n$-dimensional Novikov algebra with a fixed $\bC$-linear basis $\{e_1, \ldots, e_n\}$.  The multiplication $\mu$ is completely determined by its $n \times n$ \emph{characteristic matrix} $M(\mu) = (e_ie_j)$, whose $(i,j)$-entry is
\[
\mu(e_i,e_j) = e_ie_j = \sum_{k=1}^n d^{ij}_ke_k,
\]
where the $d^{ij}_k$ are the structure scalars of $\mu$.  If $\alpha \colon A \to A$ is an algebra morphism and $\mu_\alpha = \alpha \circ \mu$ is the associated Hom-Novikov product (Theorem ~\ref{thm:deform}), then its characteristic matrix is defined similarly as $M(\mu_\alpha) = (\alpha(e_ie_j))$, i.e., its $(i,j)$-entry is $\sum_{k=1}^n d^{ij}_k\alpha(e_k)$.

We denote a linear map $\alpha \colon A \to A$ by its $n \times n$-matrix $M(\alpha)$ with respect to the basis $\{e_1, \ldots , e_n\}$.  In other words, $M(\alpha) = (a_{ij})$, where the scalars $a_{ij}$ are defined by $\alpha(e_i) = \sum_{k=1}^n a_{ki}e_k$ for $1 \leq i \leq n$.

Note that the $0$-map, $\alpha(e_i) = 0$ for all $i$, is always an algebra morphism, and the characteristic matrix $M(\mu_\alpha) = (0)$.  In other words, the $0$-map gives rise to the trivial Hom-Novikov algebra $(A,\mu_\alpha=0,\alpha=0)$, regardless of what $\mu$ is.  Therefore, in what follows, \emph{we will omit mentioning the $0$-map explicitly to avoid unnecessary repetitions}.

A linear map $\alpha$ on $(A,\mu)$ is an algebra morphism if and only if it commutes with $\mu$, i.e., $\mu \circ \alpha^{\otimes 2} = \alpha \circ \mu$.  To determine which linear maps $\alpha$ are algebra morphisms, we apply both $\mu \circ \alpha^{\otimes 2}$ and $\alpha \circ \mu$ to the $n^2$ $\bC$-linear basis elements $\{e_1 \otimes e_1, e_1 \otimes e_2, \ldots, e_n \otimes e_n\}$ of $A^{\otimes 2}$.  Using elementary algebra, we then solve the resulting $n^2$ simultaneous equations
\[
(\mu \circ \alpha^{\otimes 2})(e_i \otimes e_j) = (\alpha \circ \mu)(e_i \otimes e_j)
\]
for the entries in $M(\alpha)$.  These simultaneous equations are tedious, but they are not difficult to solve when $A$ has dimensions two or three.

When the Novikov algebra $(A,\mu)$ has dimension two, all of its algebra morphisms are listed in Table ~\ref{2dtable}, where $\mu$, $\alpha$, and $\mu_\alpha = \alpha \circ \mu$ are presented as their respective matrices.  The first column in Table ~\ref{2dtable} lists all the (isomorphism classes of) $2$-dimensional Novikov algebras, which are classified in \cite[Table 1]{bm3}.  The notations $(T1)$, $(T2)$, etc. are also taken from \cite{bm3}.  The second column in Table ~\ref{2dtable} lists all the algebra morphisms of the Novikov algebra, and the third column lists the corresponding Hom-Novikov products $\mu_\alpha$.  The scalars $a, b, a_1, a_2, b_1, b_2$, and $\lambda$ run through $\bC$, unless explicitly stated otherwise under the respective matrices.

\begin{center}
\begin{longtable}{lll}
\caption{Algebra morphisms on $2$-dimensional Novikov algebras and their associated Hom-Novikov products.}\label{2dtable}\\

\hline Novikov algebra $M(\mu)$ & Algebra morphisms $M(\alpha)$ & Hom-Novikov products $M(\mu_\alpha)$ \\ \hline
\endfirsthead

\multicolumn{3}{c}{Table \ref{2dtable} continued} \\ \hline
Novikov algebra $M(\mu)$ & Algebra morphisms $M(\alpha)$ & Hom-Novikov products $M(\mu_\alpha)$ \\ \hline
\endhead

$(T1) = \Tone$ & $\Tonea$ & $\Tone$ \\
$(T2) = \Ttwo$ & $\Ttwoa$ & $\TtwoaM$ \\
$(T3) = \Tthree$ & $\Tthreea$, $\Tthreeb$ & $\Tone$, $\TthreebM$ \\
 & $b_1b_2 = 0$  \quad $a \not= 0$ & \\
$(N1) = \None$ & $\Nonea$ & $\NoneaM$ \\
& $a_i, b_i \in \{0,1\}$ and & \\
& $a_ib_i = 0$ for $i=1,2$ & \\
$(N2) = \Ntwo$ & $\Ntwoa$ & $\NtwoaM$ \\
& $a \in \{0,1\}$ & \\
$(N3) = \Nthree$ & $\Nthreea$ & $\NthreeaM$ \\
$(N4) = \Nfour$ & $\Nfoura$ & $\NfouraM$ \\
$(N5) = \Nfive$ & $\Nfivea$ & $\NfiveaM$ \\
$(N6) = \Nsix$ & $\Nsixa$  & $\NsixaM$ \\
\hspace{1.3cm} $\lambda \not= 0,1$ & ~$a \not= 0$ & \\ \hline
\end{longtable}
\end{center}

We make the following observations from Table ~\ref{2dtable}.
\begin{enumerate}
\item
$(T1)$ and $(N2)$ do not have non-trivial (i.e., non-zero and non-identity) deformations into Hom-Novikov algebras via algebra morphisms, while the other seven $2$-dimensional Novikov algebras do.
\item
All the $2$-dimensional Novikov algebras have infinitely many algebra morphisms, except $(N1)$, which has exactly nine, including the $0$-map.
\item
All the $2$-dimensional Novikov algebras have infinitely many non-identity algebra automorphisms, except $(N1)$, which has only one with $\alpha(e_1) = e_2$, $\alpha(e_2) = e_1$.
\item
For $(N5)$ and $(N6)$, all the non-zero algebra morphisms are invertible, since the determinants for $M(\alpha)$ in these cases are always non-zero.
\item
For each of $(N2)$, $(N3)$, $(N5)$, and $(N6)$, the algebra morphisms commute with one another.
\item
$(T1)$, $(T2)$, $(T3)$, and $(N1)$ have non-zero nilpotent algebra morphisms, while the other five do not.
\end{enumerate}

\section{Examples of $3$-dimensional Hom-Novikov algebras}
\label{sec:3d}

We continue to work over the ground field $\bC$ of complex numbers.  The purpose of this section is to classify all the algebra morphisms $\alpha$ on all the $3$-dimensional Novikov algebras $(A,\mu)$ (i.e., $\alpha \circ \mu = \mu \circ \alpha^{\otimes 2}$).  We use the classification of $3$-dimensional Novikov algebras in \cite{bm3}.  These Novikov algebras are stated in the first columns of the tables below.  From Theorem ~\ref{thm:deform}, we then obtain their corresponding $3$-dimensional Hom-Novikov algebras $A_\alpha = (A,\mu_\alpha = \alpha \circ \mu,\alpha)$.

We use the same notations and conventions as in the previous section.  In particular, with respect to a fixed $\bC$-linear basis $\{e_1,e_2,e_3\}$ of $A$, the multiplications $\mu$ and $\mu_\alpha = \alpha \circ \mu$ and the map $\alpha \colon A \to A$ are presented as their respective $3 \times 3$ matrices $M(\mu) = (e_ie_j)$, $M(\mu_\alpha) = (\alpha(e_ie_j))$, and $M(\alpha)$.  To compute the desired algebra morphisms, we use the same elementary method as described in the previous section.  More precisely, for each $3$-dimensional Novikov algebra $(A,\mu)$, we solve the nine simultaneous equations
\[
(\mu \circ \alpha^{\otimes 2})(e_i \otimes e_j) = (\alpha \circ \mu)(e_i \otimes e_j)
\]
with $1 \leq i,j \leq 3$ for the entries in $M(\alpha)$.  The computation is tedious but conceptually elementary.

In the tables below, the scalars $a$, $b$, $c$, $a_1, \ldots , c_3$, and $\lambda$ run through $\bC$, unless explicitly stated otherwise underneath the matrix (e.g., case $(A7)$).  In case $(A6)$, the symbol $\sqrt{\lambda}$ denotes either one of the two square roots of $\lambda$.  The $0$ matrix is denoted by $(0)$.  To make the tables easier to read, we sometimes state the algebra morphisms with slight overlap.  For example, in case $(A3)$, if one sets $c_2 = b_3 = b_2  = 0$, then all four types of algebra morphisms coincide.  As in the previous section, the $0$-map on $A$ will not be stated separately or explicitly, since it is always an algebra morphism with trivial associated Hom-Novikov product.

The classification of Novikov algebras $(A,\mu)$ of dimension three given in \cite{bm3} is divided into five classes, $A$ through $E$.  Together with their associated Hom-Novikov products $\mu_\alpha = \alpha \circ \mu$ (Theorem ~\ref{thm:deform}), the classifications of algebra morphisms on $3$-dimensional Novikov algebras are listed in the five tables below, one for each of the five classes.

\small


\begin{center}
\begin{longtable}{lll}
\caption{Algebra morphisms on $3$-dimensional Novikov algebras of type $A$ and their associated Hom-Novikov products.}\label{3dtableA}\\

\hline Novikov algebra $M(\mu)$ & Algebra morphisms $M(\alpha)$ & Hom-Novikov products $M(\mu_\alpha)$ \\ \hline
\endfirsthead

\multicolumn{3}{c}{Table \ref{3dtableA}: Continued.} \\ \hline
Novikov algebra $M(\mu)$ & Algebra morphisms $M(\alpha)$ & Hom-Novikov products $M(\mu_\alpha)$ \\ \hline
\endhead

\hline
\endfoot

$(A1) = (0)$ & $\Aonea$ & $(0)$ \\
$(A2) = \Atwo$ & $\Atwoa$ & $\AtwoaM$ \\
$(A3) = \Athree$ & $\Athreea$ & $(0)$ \\
& $\Athreeb$ & $\AthreebM$ \\
& $\Athreec$ & $\AthreecM$ \\
& $\Athreed$ & $(0)$ \\
$(A4) = \Afour$ & $\Afoura$ & $\AfouraM$ \\
$(A5) = \Afive$ & $\Afivea$ & $\AfiveaM$ \\
$(A6) = \Asix$ & $\Asixa$ & $(0)$ \\
& \quad $\lambda = 0$ & \\
& $\Asixb$ & $\AsixbM$ \\
& \quad $\lambda = 0$ & \\
& $\Asixc$ & $(0)$ \\
& \quad $\lambda \not= 0$ & \\
& $\Asixd$ & $\AsixdM$ \\
& \quad $\lambda \not= 0$ & \\
& $\Asixe$ & $\AsixeM$ \\
& \quad $\lambda \not= 0$ & \\
& $\Asixf$ & $(0)$ \\
& \quad $\lambda \not= 0$ & \\
$(A7) = \Aseven$ & $\Asevena$ & $\AsevenaM$ \\
\multicolumn{1}{c}{$\lambda \not= 1$} & & \\
$(A8) = \Aeight$ & $\Aeighta$ & $\AeightaM$ \\
$(A9) = \Anine$ & $\Aninea$ & $(0)$ \\
& $\Anineb$ & $(0)$ \\
& $\Aninec$ & $\AninecM$ \\
$(A10) = \Aten$ & $\Atena$ & $\AtenaM$ \\
& $\Atenb$ & $(0)$ \\
& $\Atenc$ & $\AtencM$ \\
$(A11) = \Aeleven$ & $\Aelevena$ & $(0)$ \\
\multicolumn{1}{c}{$0 \not= |\lambda| \leq 1$} & & \\
& $\Aelevenb$ & $(0)$ \\
& \quad $c \not= 0,1$ & \\
& $\Aelevenc$ & $\AelevencM$ \\
& \quad $\lambda \not= 1$ & \\
& $\Aelevend$ & $\AelevendM$ \\
& \quad $\lambda \not= 1$ & \\
& $\Aelevene$ & $\AeleveneM$ \\
& \quad $\lambda = -1$ & \\
& $\Aelevenf$ & $\AelevenfM$ \\
& \quad $\lambda = 1$ & \\
& $\Aeleveng$ & $\AelevengM$ \\
& \quad $\lambda \not= 1$ & \\
$(A12) = \Atwelve$ & $\Atwelvea$ & $(0)$ \\
& $c_1 = c_2 = 0$ or $c_3 = 0$ & \\
& $\Atwelvec$ & $\AtwelvecM$ \\
$(A13) = \Athirteen$ & $\Athirteena$ & $(0)$ \\
& \quad $c_1c_3 = 0$ & \\
& $\Athirteenb$ & $\AthirteenbM$ \\
& $\Athirteenc$ & $\AthirteencM$ \\
\end{longtable}
\end{center}


\normalsize
The case $(B0)$ below is discussed in \cite[4.1]{bm3}, and is not stated with the other $3$-dimensional Novikov algebras of type $B$ in \cite[Table 3]{bm3}.
\small

\begin{center}
\begin{longtable}{lll}
\caption{Algebra morphisms on $3$-dimensional Novikov algebras of type $B$ and their associated Hom-Novikov products.}\label{3dtableB}\\

\hline Novikov algebra $M(\mu)$ & Algebra morphisms $M(\alpha)$ & Hom-Novikov products $M(\mu_\alpha)$ \\ \hline
\endfirsthead

\multicolumn{3}{c}{Table \ref{3dtableB}: Continued.} \\ \hline
Novikov algebra $M(\mu)$ & Algebra morphisms $M(\alpha)$ & Hom-Novikov products $M(\mu_\alpha)$ \\ \hline
\endhead

\hline
\endfoot

$(B0) = \Bzero$ & $\Bzeroa$ & $\BzeroaM$ \\
& $a_ib_i = a_ic_i = b_ic_i = 0$, & \\
& $a_i, b_i, c_i \in \{0,1\}$ & \\
& for $i = 1,2,3$ & \\
$(B1) = \Bone$ & $\Bonea$ & $\BoneaM$ \\
& $c_2, c_3 \in \{0,1\}$ & \\
& $\Boneb$ & $\BonebM$ \\
& $b_2, b_3 \in \{0,1\}$ \\
& $\Bonec$ & $\BonecM$ \\
& $b, c \in \{0,1\}$ & \\
& $\Boned$ & $\BonedM$ \\
& $b, c \in \{0,1\}$ & \\
$(B2) = \Btwo$ & $\Btwoa$ & $\BtwoaM$ \\
& $b_2, b_3, c \in \{0,1\}$ and $b_2c = 0$ & \\
& $\Btwob$ & $\BtwobM$ \\
& $b, c \in \{0,1\}$ and $bc = 0$ & \\
$(B3) = \Bthree$ & $\Bthreea$ & $\BthreeaM$ \\
& $b, c \in \{0,1\}$ and $bc = 0$ & \\
& $\Bthreeb$ & $\BthreebM$ \\
& $b_2, c \in \{0,1\}$ and $b_2c = 0$ & \\
& $\Bthreec$ & $\BthreecM$ \\
& $b, c_2 \in \{0,1\}$ and $bc_2 = 0$ & \\
$(B4) = \Bfour$ & $\Bfoura$ & $\BfouraM$ \\
& $b, c \in \{0,1\}$ and $bc = 0$ & \\
& $\Bfourb$ & $\BfourbM$ \\
& $b, c_2 \in \{0,1\}$ and $bc_2 = 0$ & \\
$(B5) = \Bfive$ & $\Bfivea$ & $\BfiveaM$ \\
\multicolumn{1}{c}{$\lambda \not= 0,1$} & & \\
& $\Bfiveb$ & $\BfivebM$ \\
& $\Bfivec$ & $\BfivecM$ \\
& $\Bfived$ & $\BfivedM$ \\
& $\Bfivee$ & $\BfiveeM$ \\
& $\Bfivef$ & $\BfivefM$ \\
& $\Bfiveg$ & $\BfivegM$ \\
& $\Bfiveh$ & $\BfivehM$ \\
\end{longtable}
\end{center}


\begin{center}
\begin{longtable}{lll}
\caption{Algebra morphisms on $3$-dimensional Novikov algebras of type $C$ and their associated Hom-Novikov products.}\label{3dtableC}\\

\hline Novikov algebra $M(\mu)$ & Algebra & Hom-Novikov products $M(\mu_\alpha)$ \\
& morphisms $M(\alpha)$ & \\ \hline
\endfirsthead

\multicolumn{3}{c}{Table \ref{3dtableC}: Continued.} \\ \hline
Novikov algebra $M(\mu)$ & Algebra & Hom-Novikov products $M(\mu_\alpha)$ \\
& morphisms $M(\alpha)$ & \\ \hline
\endhead

\hline
\endfoot

$(C1) = \Cone$ & $\Conea$ & $\ConeaM$ \\
& \quad $c \in \{0,1\}$ & \\
$(C2) = \Ctwo$ & $\Ctwoa$ & $(0)$ \\
& $\Ctwob$ & $\CtwobM$ \\
$(C3) = \Cthree$ & $\Cthreea$ & $(0)$ \\
& $\Cthreeb$ & $\CthreebM$ \\
$(C4) = \Cfour$ & $\Cfoura$ & $(0)$ \\
& $\Cfourb$ & $\CfourbM$ \\
$(C5) = \Cfive$ & $\Cfivea$ & $(0)$ \\
\multicolumn{1}{c}{$\lambda \not= 0,1$} & & \\
& $\Cfiveb$ & $\CfivebM$ \\
$(C6) = \Csix$ & $\Csixa$ & $\CsixaM$ \\
$(C7) = \Cseven$ & $\Csevena$ & $\CsevenaM$ \\
$(C8) = \Ceight$ & $\Ceighta$ & $\CeightaM$ \\
$(C9) = \Cnine$ & $\Cninea$ & $\CnineaM$ \\
\multicolumn{1}{c}{$\lambda \not= 0,1$} & & \\
$(C10) = \Cten$ & $\Ctena$ & $\CtenaM$ \\
\multicolumn{1}{c}{$\lambda \not= 1$} & $\lambda a_2 = 0 = \lambda c_1$ & \\
$(C11) = \Celeven$ & $\Celevena$ & $\CelevenaM$ \\
$(C12) = \Ctwelve$ & $\Ctwelvea$ & $\CtwelveaM$ \\
\multicolumn{1}{c}{$\lambda \not= 0,1$} & & \\
$(C13) = \Cthirteen$ & $\Cthirteena$ & $\CthirteenaM$ \\
\multicolumn{1}{c}{$\lambda, \eta \not= 0,1$} & \quad if $\eta \not= \lambda$ & \\
& $\Cthirteenb$ & $\CthirteenbM$ \\
& \quad if $\eta = \lambda$ & \\
$(C14) = \Cfourteen$ & $\Cfourteena$ & $\CfourteenaM$ \\
$(C15) = \Cfifteen$ & $\Cfifteena$ & $\CfifteenaM$ \\
\multicolumn{1}{c}{$\lambda \not= 0,1$} & & \\
$(C16) = \Csixteen$ & $\Csixteena$ & $\CsixteenaM$ \\
$(C17) = \Cseventeen$ & $\Cseventeena$ & $\CseventeenaM$ \\
$(C18) = \Ceighteen$ & $\Ceighteena$ & $\CeighteenaM$ \\
$(C19) = \Cnineteen$ & $\Cnineteena$ & $\CnineteenaM$ \\
\end{longtable}
\end{center}


\begin{center}
\begin{longtable}{lll}
\caption{Algebra morphisms on $3$-dimensional Novikov algebras of type $D$ and their associated Hom-Novikov products.}\label{3dtableD}\\

\hline Novikov algebra $M(\mu)$ & Algebra & Hom-Novikov products $M(\mu_\alpha)$ \\
& morphisms $M(\alpha)$ & \\ \hline
\endfirsthead

\multicolumn{3}{c}{Table \ref{3dtableD}: Continued.} \\ \hline
Novikov algebra $M(\mu)$ & Algebra & Hom-Novikov products $M(\mu_\alpha)$ \\
& morphisms $M(\alpha)$ & \\ \hline
\endhead

\hline
\endfoot

$(D1) = \Done$ & $\Donea$ & $\DoneaM$ \\
& \quad $c \in \{0,1\}$ & \\
$(D2) = \Dtwo$ & $\Dtwoa$ & $\DtwoaM$ \\
$(D3) = \Dthree$ & $\Dthreea$ & $\DthreeaM$ \\
& \quad $a_1 \in \{0,1\}$ & \\
$(D4) = \Dfour$ & $\Dfoura$ & $\DfouraM$ \\
$(D5) = \Dfive$ & $\Dfivea$ & $\DfiveaM$ \\
$(D6) = \Dsix$ & $\Dsixa$ & $\DsixaM$ \\
\multicolumn{1}{c}{$\lambda \not= \frac{1}{2},1$} & & \\
& $\Dsixb$ & $\DsixbM$ \\
& \quad $\lambda = 0$ & \\
\end{longtable}
\end{center}


\begin{center}
\begin{longtable}{lll}
\caption{Algebra morphisms on the $3$-dimensional Novikov algebra of type $E$ and their associated Hom-Novikov products.}\label{3dtableE}\\

\hline Novikov algebra $M(\mu)$ & Algebra morphisms $M(\alpha)$ & Hom-Novikov products $M(\mu_\alpha)$ \\ \hline
\endfirsthead

\multicolumn{3}{c}{Table \ref{3dtableE}: Continued.} \\ \hline
Novikov algebra $M(\mu)$ & Algebra morphisms $M(\alpha)$ & Hom-Novikov products $M(\mu_\alpha)$ \\ \hline
\endhead

\hline
\endfoot

$(E1) = \Eone$ & $\Eonea$ & $\EoneaM$ \\
& \quad $c \in \{0,1\}$ & \\
& $\Eoneb$ & $\EonebM$ \\
& \quad $c \in \{0,1\}$ & \\
& $\Eonec$ & $\EonecM$ \\
& \quad $c \in \{0,1\}$ & \\
\end{longtable}
\end{center}

\normalsize

We end this paper with some concluding comments and observations.

\begin{enumerate}
\item
Some non-isomorphic Novikov algebras have isomorphic Hom-Novikov deformations with non-trivial Hom-Novikov products.  For example, the second type of algebra morphisms on $(B1)$ with $a=0$ coincides with the first type of algebra morphisms on $(B2)$ with $c=0$.  Moreover, their corresponding Hom-Novikov products have the same non-zero characteristic matrices.  Therefore, they are isomorphic as non-trivial Hom-Novikov algebras, even though the original Novikov algebras $(B1)$ and $(B2)$ are not isomorphic.  The same phenomenon also happens in dimension two.  For example, the $2$-dimensional Novikov algebras $(N1)$ (with $a_2 = b_1 = b_2 = 0$ in its algebra morphisms) and $(N2)$ (with $b=0$ in its algebra morphisms) can be deformed into isomorphic Hom-Novikov algebras that have non-trivial Hom-Novikov products.  There are several other such pairs in the tables above.
\item
Related to the previous observation, note that we have \emph{not} classified Hom-Novikov algebras in low dimensions.  Such a classification remains an open question.  Perhaps the methods in \cite{fg2,gohr} for the classification of Hom-associative algebras can be adapted to Hom-Novikov algebras.
\item
Certain Novikov algebras are closely related to the geometry of Lie groups \cite{bm3,bm6,perea}.  Is there a similar relationship between Hom-Novikov algebras and some twisting of Lie groups?
\end{enumerate}


\end{document}